\documentclass[a4paper,11pt]{article}
\usepackage{amssymb}
\usepackage{amsfonts}
\usepackage{amsmath}
\usepackage{mathrsfs}
\usepackage{graphicx}
\usepackage{amsbsy}
\usepackage{theorem}
\usepackage{color}
\usepackage[bookmarks=false]{hyperref}
\usepackage{tikz}
\usepackage[normalem]{ulem}
\usepackage[latin1]{inputenc}
\newtheorem{theorem}{Theorem}[section]
\newtheorem{lemma}[theorem]{Lemma}

\newtheorem{definition}[theorem]{Definition}
\newtheorem{proposition}[theorem]{Proposition}
\newtheorem{remark}[theorem]{Remark}

\def\thetheorem{\thesection.\arabic{theorem}}
\def\thesection{\arabic{section}}

\def\theequation {\thesection.\arabic{equation}}
\def\beq{\begin{equation}\displaystyle}
\def\eeq{\end{equation}}
\def\bel{\begin{equation} \displaystyle \begin{array}{l} }
\def\eel{\end{array} \end{equation} }
\def\bell{\begin{equation} \displaystyle \begin{array}{ll}  }
\def\eell{\end{array} \end{equation} }

\def\bea{\begin{eqnarray}}
\def\eea{\end{eqnarray} }
\def\bean{\begin{eqnarray*}}
\def\eean{\end{eqnarray*} }
\newenvironment{proof}{\noindent{\bf Proof.~}}
{{\mbox{}\hfill {\small \fbox{}}\\}}
\catcode`@=11
\renewcommand\appendix{\bigskip {\noindent \Large \bf Appendix}
  \setcounter{section}{0}%
  \setcounter{subsection}{0}%
\setcounter{equation}{0}%
\setcounter{theorem}{0}%
\def\thetheorem{A.\arabic{theorem}}
\def\theequation {A.\arabic{equation}}}
\catcode`@=12

\newcommand{\dv}{\mathop{\rm div}\nolimits}

\def\NN{\mathbb{N}}

\def\RR{\mathbb{R}}

\def\ZZ{\mathbb{Z}}

\def\ds{\displaystyle}

\def\bs{\bigskip}

\def\pa{\partial}

\def\calM{\mathcal{M}}
\def\calW{\mathcal{W}}
\def\calP{\mathcal{P}}

\def\smes{\mathcal{S}_\mathcal{M}}
\def\Dpetit{{\mbox{\tiny $\Delta$}}}
\def\achapo{\widehat{a}}
\def\bchapo{\widehat{b}}

\DeclareMathOperator{\sgn}{sgn}

\begin{document}

\title{One-dimensional aggregation equation after blow up: existence, uniqueness and numerical simulation}

\author{Fran\c{c}ois James \thanks{Math\'ematiques -- Analyse, Probabilit\'es, Mod\'elisation -- Orl\'eans (MAPMO),
Universit\'e d'Orl\'eans \& CNRS UMR 7349}
\thanks{F\'ed\'eration Denis Poisson, Universit\'e d'Orl\'eans \& CNRS FR 2964,45067 Orl\'eans Cedex 2, France}
and Nicolas Vauchelet  \thanks{Sorbonne Universit\'es, UPMC Univ Paris 06, Laboratoire Jacques-Louis Lions UMR CNRS 7598, Inria, F-75005, Paris, France}
\thanks{The second author is supported by the ANR blanche project Kibord: ANR-13-BS01-0004 funded by the French Ministry of Research}}

\date{}

\maketitle

\begin{abstract}
The nonlocal nonlinear aggregation equation in one space dimension is investigated. 
In the so-called attractive case smooth solutions blow up in finite time, so that weak measure solutions are introduced.
The velocity involved in the equation becomes discontinuous, and a particular care has to be paid to its definition
as well as the formulation of the corresponding flux. When this is done, the notion of duality solutions allows to
obtain global in time existence and uniqueness for measure solutions. An upwind finite volume scheme is also analyzed,
and the convergence towards the unique solution is proved. Numerical examples show the 
 dynamics of the solutions after the blow up time.
\end{abstract}

\bigskip



\section{Introduction}

This paper presents a survey of several results obtained by the authors concerning existence, uniqueness and numerical simulation
of measure solutions for the one-dimensional aggregation equation in the attractive case. 
This equation describes aggregation phenomena in a population of individuals
interacting under a continuous potential $W\,:\,\RR\to\RR$. If $\rho$ denotes the density of individuals, its dynamics is modelled by a 
nonlocal nonlinear conservation equation
	\begin{equation}\label{EqInter}
\pa_t\rho + \pa_x\big(a(\pa_xW*\rho) \rho\big) = 0, \qquad t>0,\quad x\in\RR.
	\end{equation}
This equation is complemented with the initial condition $\rho(0,x)=\rho^{ini}$. Here 
 $a\,:\,\RR\to\RR$ is a smooth 
given function which depends on the actual model under consideration.
This model appears in many applications in physics and population dynamics.
It is used for instance in the framework of granular media \cite{benedetto}, in the description of crowd motion \cite{pieton,maury},
and in the description of the collective motion of cells or bacteria
\cite{okubo,dolschmeis,filblaurpert,NoDEA}, and the references therein.
In many of these examples, the potential $W$ has a singularity at the origin.
Due to this weak regularity, finite time blow up of regular solutions occurs 
and has caught the attention of several authors \cite{Li,BV,Bertozzi1}.
The context of weak measure solution is natural here, because of this blow up property as well as the conservative structure
 of the aggregation equation \eqref{EqInter},
which gives rise to a bound on the mass of measure solutions.

In this paper, we focus on these singular attractive potentials, sometimes called mild singular. More precisely, we introduce the following notion.
	\begin{definition}[pointy potential]\label{pointyPot}
The interaction potential $W$ is said to be an {\it attractive 
pointy potential} if it satisfies the following assumptions:
	\begin{equation}\label{hyp1}\begin{array}{c}
W \mbox{ is Lipschitz continuous, } W(x)=W(-x),\ W(0)=0, \\[2mm]
\mbox{ and  } W \mbox{ is } \lambda\mbox{-concave for some } \lambda\geq 0,
\mbox{ i.e. } W(x)-\frac{\lambda}{2}x^2 \mbox{ is concave}.
	\end{array}\end{equation}
	\end{definition}

Depending upon the applications, the function $a$ may be linear ($a(x)=x$) or nonlinear. 
In what follows, we shall refer to the {\it linear case} when $a={\rm id}$ and the potential $W$
satisfies \eqref{hyp1}. In the {\it nonlinear case} additional assumptions have to be made. First, to 
obtain an attractive model, $a$ has to be nondecreasing. This is related to the so-called one-sided Lipschitz estimates, see Section \ref{sec:nonlin} below
for details. Therefore we consider the following set of assumptions on the velocity field:
	\begin{equation}\label{hyp_a}
a\in C^1(\RR),\quad 0\leq a' \leq \alpha, \quad \alpha>0.
	\end{equation}
Unfortunately, in this case, the class of admissible potentials has to be reduced, namely we are limited to potentials $W$ such that
	\begin{equation}\label{hyp2}\begin{array}{c}
W\in C^2(\RR\setminus\{0\}) \mbox{ satisfies \eqref{hyp1} and there exists $w$ 
continuous, } \|w\|_{L^1(\RR)} = w_0  \\[2mm]
\mbox{such that }W''= -\delta_0 + w \mbox{ holds in the distributional sense}.
\end{array}
	\end{equation}
The above assumptions include classical functions $W$ such as $W(x)=-|x|$ or $W(x)=e^{-|x|}-1$.
As we shall see, each case deserves its own definition of the velocity.

Several authors have studied existence of global in time weak measure solution for the 
aggregation equation. In \cite{Carrillo}, global existence of weak measure solutions in the linear case, that is
for $W$ satisfying \eqref{hyp1}, in $\RR^d$ for any dimension $d\geq 1$ has been obtained
using the gradient flow structure of this problem.
In fact, for the aggregation equation in the case $a={\rm id}$, we can define the interaction energy by
	$$
\calW(\rho)= \frac 12\int_{\RR^d\times\RR^d} W(x-y)\,\rho(dx)\rho(dy).
	$$
Then a gradient flow solution $\mu$ in the Wasserstein space is defined as a solution 
in the sense of distributions of the continuity equation 
	$$
\pa_t \mu + \dv\big(v \mu \big)=0, \qquad v\in \pa^0\calW(\mu),
	$$
where $\pa^0\calW(\mu)$ denotes the element of minimal norm in $\pa\calW(\mu)$,
which is the subdifferential of $\calW$ at the point $\mu$, see \cite{Ambrosio} for more details.
Such a solution is constructed by performing the JKO scheme \cite{JKO}.
However this approach cannot be applied in the nonlinear case that is under assumptions \eqref{hyp_a}-\eqref{hyp2}
and there is, up to our knowledge, no numerical result based on this approach 
allowing to recover the dynamics of the solution after blow up.

An alternative strategy has been proposed by the authors, which consists in interpreting \eqref{EqInter} as a conservative transport
equation with velocity $a(\pa_xW*\rho)$. Since solutions blow up in finite time, eventually $\rho$ become measure-valued, and 
care has to be taken of the product $a(\pa_xW*\rho)\rho$: typically Dirac masses may appear and the velocity becomes discontinuous
precisely at their location. Hence this  requires the use of tools which have been developped for advection equations with discontinuous coefficients:
pushforward by a generalized flow \cite{PoupaudRascle,Bianchini}, or duality solutions \cite{bj1}. This paper will make use of the latter notion, 
which is recalled in the next Section. The first application to the aggregation equation was done in the particular case of chemotaxis in 
 \cite{NoDEA}.  It has been extended later to more general aggregation equations in both the linear and nonlinear cases
in \cite{GF_dual}. The main drawback of this method is that it is presently limited to one space dimension. In the linear multidimensional case
(as in \cite{Carrillo}), the pushforward method has been successfully applied in \cite{CJLV}. We emphasize that in all cases the definition of the 
velocity and of its product with the measure $\rho$ has to be very carefully treated, as it is a key ingredient to prove
the uniqueness of solutions.

Numerical simulations of solutions to \eqref{EqInter} before the blow up time 
has been investigated in \cite{CCH} with a finite volume method but no convergence result has been obtained,
and in \cite{CB} thanks to a particle method.
However, the dynamics of the solutions after the blow up time is not recovered in these works.
Then, in \cite{sinum}, a finite volume scheme of Lax-Friedrichs type has been proposed and analyzed. 
This scheme has been designed in order to recover the dynamics of the solution after blow up time. 
In this paper, we study another finite volume scheme, based on an upwind approach.
As in \cite{sinum}, the convergence of the scheme is proved and numerical simulations showing different 
blow up profiles are proposed.

The outline of the paper is the following. The next Section is devoted to the definition of duality 
solution to equation \eqref{EqInter}. We first recall useful results on duality solutions for
transport equation. Then the definition of duality solution for the problem at hand is
defined in subsection \ref{sec22}. A particular attention is given to the definition of the 
 flux and velocity in both sets of assumptions.
Section \ref{sec:lin} is devoted to the proof of existence and uniqueness of weak measure solutions
in the linear case ($a={\rm id}$ and $W$ satisfying \eqref{hyp1}).
The nonlinear case, that is assumptions \eqref{hyp_a}-\eqref{hyp2}, is studied in Section \ref{sec:nonlin}.
Section \ref{sec:lin} and \ref{sec:nonlin} summarize the main results of the
articles \cite{NoDEA,GF_dual}.
Finally the numerical resolution of the problem is proposed in Section \ref{num}.
The  convergence of an upwind-type finite volume scheme is obtained in Theorem \ref{th:convmacro}.
Numerical illustrations showing different behaviours of solutions after blow up for different choices
 of the interaction potential are provided in subsection \ref{sec:numsim}

\section{Duality solutions}

We will make use of the notations $C_0(\RR)$ for the set of continuous functions that vanish at infinity,
$\calM_b(\RR)$ for the space of finite measures on $\RR$.
For $\rho\in \calM_b(\RR)$, its total mass is denoted $|\rho|(\RR)$.
This space will be always endowed with the weak topology $\sigma(\calM_b,C_0)$
and we denote $\smes :=C([0,T];{\mathcal M}_b(\RR)-\sigma({\mathcal M}_b,C_0))$.
Since we focus on scalar conservation laws, we can assume without loss
of generality that the total mass of the system is scaled to $1$.
Indeed, if the total mass is $M\neq 1$, then we rescale the density by introducing
$\rho/M$, it suffices to change the definition of the function $a$ 
by introducing $\widetilde{a}(x)=a(Mx)$ which will always satisfies \eqref{hyp_a}.
Thus we will work in some space of probability measures, namely 
the Wasserstein space of order $q\geq 1$, which is the space
of probability measures with finite order $q$ moment:
$$
\calP_q(\RR) = \left\{\mu \mbox{ nonnegative Borel measure}, \mu(\RR)=1, \int |x|^q \mu(dx) <\infty\right\}.
$$

\subsection{Duality solutions for linear transport equation}\label{duality}
We consider the conservative transport equation
\begin{equation}\label{eq.conserve}
\partial_t\rho + \pa_x\big(b(t,x)\rho\big) = 0, \qquad (t,x)\in(0,T)\times\RR,
\end{equation}
where $b$ is a given bounded Borel function. Since no regularity is assumed for $b$, 
solutions to \eqref{eq.conserve} eventually are measures in space. A convenient tool to handle this
is the notion of duality solutions, which are defined as weak solutions, the test functions 
being Lipschitz solutions to the backward linear transport equation
	\begin{eqnarray}
& & \ds \partial_t p + b(t,x) \partial_x p = 0, \label{(3)} \qquad
p(T,.) = p^T \in {\rm Lip}(\RR).
	\end{eqnarray}
In fact, a formal computation shows that 
$\frac{d}{dt}\left(\int_{\RR}p(t,x)\rho(t,dx)\right) = 0$,
which defines the duality solutions for suitable $p$.

It is quite classical that a sufficient condition to ensure existence for \eqref{(3)} is that 
the velocity field be compressive, in the following sense:
\begin{definition}
We say that the function $b$ satisfies the one-sided Lipschitz (OSL) condition if
\begin{equation}\label{OSLC}
\partial_x b(t,.)\leq \beta(t)\qquad\mbox{for $\beta\in L^1(0,T)$, in the distributional sense}.
\end{equation}
\end{definition}
However, to have uniqueness, we need to restrict ourselves to {\it reversible} solutions of (\ref{(3)}): let ${\mathcal L}$ 
denote the set of Lipschitz continuous solutions to (\ref{(3)}), and define
the set ${\mathcal E}$ of {\it exceptional solutions} by
	$$
{\mathcal E} = \Big\{p\in{\mathcal L} \mbox{ such that } p^T \equiv 0 \Big\}.
	$$ 
The possible loss of uniqueness corresponds to the case where ${\mathcal E}$ is not reduced to $\{\rho=0\}$.
	\begin{definition}
We say that $p\in{\mathcal L}$ is a {\bf reversible solution} to (\ref{(3)}) 
if $p$ is locally constant on the set 
	$$
{\mathcal V}_e=\Big\{(t,x)\in [0,T] \times \RR;\ \exists\ p_e\in{\mathcal E},\ p_e(t,x)\not=0\Big\}.
	$$
	\end{definition}
We refer to \cite{bj1} for complete statements of the characterization and properties of reversible solutions.
Then, we can state the definition of duality solutions.
	\begin{definition}
We say that 
$\rho\in \smes := C([0,T];{\mathcal M}_b(\RR)-\sigma({\mathcal M}_b,C_0))$
is a {\bf duality solution} to (\ref{eq.conserve}) 
if for any $0<\tau\le T$, and any {\bf reversible} solution $p$ to (\ref{(3)})
with compact support in $x$,
the function $\displaystyle t\mapsto\int_{\RR}p(t,x)\rho(t,dx)$ is constant on
$[0,\tau]$.
	\label{defdual}\end{definition}

We summarize now some useful properties of duality solutions.

	\begin{theorem}(Bouchut, James \cite{bj1})\label{ExistDuality}
\begin{enumerate}
\item Given $\rho^\circ \in {\mathcal M}_b(\RR)$, under the assumptions
(\ref{OSLC}), there exists a unique $\rho \in \smes$,
duality solution to (\ref{eq.conserve}), such that $\rho(0,.)=\rho^\circ$. \\
Moreover, if $\rho^\circ$ is nonnegative, then $\rho(t,\cdot)$ is nonnegative
for a.e. $t\geq 0$. And we have the mass conservation 
	$$
|\rho(t,\cdot)|(\RR) = |\rho^\circ|(\RR), \quad \mbox{ for a.e. } t\in ]0,T[.
	$$
\item Backward flow and push-forward: the duality solution satisfies
	\begin{equation}\label{flow}
\forall\, t\in [0,T], \forall\, \phi\in C_0(\RR),\quad
\int_\RR \phi(x)\rho(t,dx) = \int_\RR \phi(X(t,0,x)) \rho^0(dx),
	\end{equation}
where the {\bf backward flow} $X$ is defined as the unique reversible
solution to
	$$
\pa_tX + b(t,x) \pa_xX = 0 \quad \mbox{ in } ]0,s[\times\RR, \qquad
X(s,s,x)=x.
	$$
\item For any duality solution $\rho$, we define the {\bf generalized flux}
  corresponding to $\rho$ by 
$b\Dpetit \rho = -\pa_t u$, where $u=\int^x \rho\,dx$.

There exists a bounded Borel function $\widehat b$, called {\bf
universal representative} of $b$, such that $\widehat b = b$
almost everywhere, $b\Dpetit \rho =\widehat b\rho$ and for any duality solution $\rho$,
	$$
\partial_t \rho + \partial_x(\widehat b\rho) = 0 \qquad \hbox{in the distributional sense.}
	$$
\item Stability: Let $(b_n)$ be a bounded sequence in
$L^\infty(]0,T[\times\RR)$, such that
$b_n\rightharpoonup b$ in $L^\infty(]0,T[\times\RR)-w\star$. Assume
$\partial_x b_n\le \beta_n(t)$, where $(\beta_n)$ is bounded in $L^1(]0,T[)$,
$\partial_x b\le\beta\in L^1(]0,T[)$.
Consider a sequence $(\rho_n)\in\smes$ of duality solutions to
	$$
\partial_t\rho_n+\partial_x(b_n\rho_n)=0\quad\hbox{in}\quad]0,T[\times\RR,
	$$
such that $\rho_n(0,.)$ is bounded in ${\mathcal M}_{b}(\RR)$, and
$\rho_n(0,.)\rightharpoonup\rho^\circ\in{\mathcal M}_{b}(\RR)$.

\noindent Then $\rho_n\rightharpoonup \rho$ in $\smes$, where $\rho\in\smes$ is the
duality solution to
	$$
\partial_t\rho+\partial_x(b\rho)=0\quad\hbox{in}\quad]0,T[\times\RR,\qquad
\rho(0,.)=\rho^\circ.
	$$
Moreover, $\widehat b_n\rho_n\rightharpoonup \widehat b\rho$ weakly in ${\mathcal M}_{b}(]0,T[\times\RR)$.
\end{enumerate}
	\end{theorem}

The set of duality solutions is clearly a vector space, but it has
to be noted that a duality solution is not {\it a priori} defined as 
a solution in the sense of distributions. 
However, assuming that the coefficient $b$ is piecewise continuous,
we have the following equivalence result:
	\begin{theorem}(Bouchut, James \cite{bj1})\label{dual2distrib}
Let us assume that in addition to the OSL condition \eqref{OSLC}, 
$b$ is piecewise continuous on $]0,T[\times\RR$ where the 
set of discontinuity is locally finite.
Then there exists a function $\bchapo$ which coincides with $b$
on the set of continuity of $b$.

With this $\bchapo$, $\rho\in \smes$ is a duality solution to 
\eqref{eq.conserve} if and only if $\pa_t\rho+\pa_x(\bchapo\rho)=0$
in ${\mathcal D}'(\RR)$. Then the generalized flux 
$b\Dpetit \rho = \bchapo \rho$. 
In particular, $\bchapo$ is a universal representative of $b$.
	\end{theorem}

This result comes from the uniqueness of solutions to the Cauchy 
problem for both kinds of solutions, see \cite[Theorem 4.3.7]{bj1}.

\subsection{Aggregation equation as a transport equation}
\label{sec22}

Equipped with this notion of solutions, we can now define duality solutions for the aggregation equation.
The idea was introduced in \cite{BJpg} in the context of pressureless gases. It was next applied to chemotaxis in \cite{NoDEA} and generalized in \cite{GF_dual}.
	\begin{definition}\label{defexist}
We say that $\rho\in \smes$
is a duality solution to \eqref{EqInter} if there exists 
$\widehat{a}_\rho\in L^\infty((0,T)\times\RR)$ and $\beta\in L^1_{loc}(0,T)$ 
satisfying $\pa_x\widehat{a}_\rho\le \beta$ in ${\mathcal D}'((0,T)\times\RR)$, 
such that for all $0<t_1<t_2<T$,
	\begin{equation}\label{rhodis}
\pa_t\rho + \pa_x(\widehat{a}_\rho\rho) = 0
	\end{equation}
in the sense of duality on $(t_1,t_2)$,
and $\widehat{a}_\rho=a(W'*\rho)$ a.e.
We emphasize that it means that the final datum for \eqref{(3)} should be at $t_2$ instead of $T$.
	\end{definition}
This allows at first to give a meaning to the notion of distributional solutions, but it turns out that uniqueness is 
a crucial issue. For that, a key point is a specific definition of the product $\widehat{a}_\rho\rho$,
which can be seen as the flux of the system.
Indeed, when concentrations occur in conservation equations, measure-valued solutions can
have Dirac deltas, which makes the velocity a BV function. A key point for the definition of the flux is to be able to
handle products of BV functions with measure-valued functions.
In the framework of the aggregation equation, a definition of the flux can be obtained 
using the dependancy of the velocity on the solution.
Let us make this point precise in both situations considered in this paper.

{\bf In the linear case} that is $a={\rm id}$ and $W$ satisfying assumptions \eqref{hyp1}, the flux is defined by
	\begin{equation}\label{achapo}
J = \widehat{a}_\rho\rho, \qquad 
\widehat{a}_\rho(t,x):= \pa^0W*\rho(t,x) = \int_{x\neq y} W'(x-y)\rho(t,dy).
	\end{equation}
This definition is motivated by the following stability result:
	\begin{lemma}\cite[Lemma 3.1]{CJLV}\label{lemstab_a}
Let us assume that $W$ satisfies \eqref{hyp1}.
Let $(W_n)_{n\in\NN^*}$ be a sequence in $C^1(\RR)$ satisfying \eqref{hyp1}
with the same constant $\lambda$ not depending on $n$ and such that
	$$
sup_{x\in \RR \setminus (-\frac 1n,\frac 1n)} \big|W'_n(x)- W'(x)\big| \leq \frac 1n, \qquad \mbox{ for all } n\in \NN^*.
	$$
If the sequence $\rho_n\rightharpoonup \rho$ weakly as measures, then
for every continuous compactly supported $\phi$, we have
	$$
\lim_{n\to +\infty} \iint_{\RR\times \RR} \phi(x) W'_n(x-y) \rho_n(dx)\rho_n(dy)
=\iint_{\RR\times \RR \setminus D} \phi(x) W'(x-y) \rho(dx)\rho(dy),
	$$
where $D$ is the diagonal of $\RR\times\RR$: $D=\{(x,x),\, x\in \RR\}$.
	\end{lemma}

{\bf In the nonlinear case} given by assumptions \eqref{hyp_a} and \eqref{hyp2}, we use 
the assumption on $W$ to obtain a definition of the flux. Indeed we can formally
take the convolution of \eqref{hyp2} by $\rho$, then multiply by $a(W'*\rho)$. Denoting by 
$A$ the antiderivative of $a$ such that $A(0)=0$  and using the chain rule we obtain formally
	\begin{equation}\label{ChainRule}
-\pa_x(A(W'*\rho))=-a(W'*\rho)W''*\rho = a(W'*\rho)(\rho-w*\rho).
	\end{equation}
Thus a natural formulation for the flux $J$ is given by
	\begin{equation}\label{DefFluxJ}
J:= -\pa_x\big(A(W'*\rho)\big)+a(W'*\rho)w*\rho.
	\end{equation}
The product $a(W'*\rho)w*\rho$ is well defined since $w*\rho$ is Lipschitz.
The function $A(W'*\rho)$ is a $BV(\RR)$ function.
Then $J$ is defined in the sense of measures.
The analogue of the stability result of Lemma \ref{lemstab_a} is verified 
since if $\rho_n \rightharpoonup \rho$, we have that $W'*\rho_n \rightharpoonup W*\rho$
a.e., which induces that in the sense of distributions $J_n$ converges to $J$.
Moreover, from the chain rule for BV functions (or Vol'pert calculus), there exists 
a function $\achapo_\rho$ such that $\achapo_\rho=a(W'*\rho)$ a.e., and $J=\achapo_\rho \rho$.
Then it can be verified (see Section 3.3 in \cite{GF_dual}) that in the case $a={\rm id}$,
$\achapo_\rho$ is given by \eqref{achapo}.

This idea of definition of the flux comes from \cite{NoDEA}, where the particular case 
$W(x)=\frac 12 e^{-|x|}-\frac 12$ appearing in chemotaxis has been treated.
An analogous situation arising in plasma physics is considered in \cite{CRAS}. 
In a similar context, other definitions of the product can
be found, see \cite{NPS} in the one-dimensional setting, 
and \cite{PoupaudDef} for a generalization in two space dimensions, 
where defect measures are used. In this latter work, more singular potentials are considered, 
but the uniqueness of the weak measure solution is not recovered.

\section{Existence and uniqueness in the linear case}\label{sec:lin}

In this section we state and prove the existence and uniqueness theorem for duality solutions to the aggregation equation \eqref{EqInter} in the
linear case, that is $a={\rm id}$ and a general pointy potential $W$ satisfying \eqref{hyp1}. 
	\begin{theorem}\cite[Theorem 3.7]{GF_dual}\label{th:duality}
Let $W$ as in \eqref{hyp1} and $a=\mbox{\rm id}$.
Assume that $\rho^{ini}\in \calP_1(\RR)$.
Then for any $T>0$, there exists a unique $\rho\in\smes$ such that $\rho(0)=\rho^{ini}$, 
$\rho(t)\in \calP_1(\RR)$ for any $t\in (0,T)$, and $\rho$ is a duality solution
to equation \eqref{EqInter} with universal representative $\widehat{a}_\rho$ in \eqref{rhodis} defined
in \eqref{achapo}.
Moreover we have $\rho = X_\# \rho^{ini}$ where $X$ is the 
backward flow corresponding to $\widehat{a}_\rho$.
	\end{theorem}
The proof of this result is splitted into several steps corresponding 
to the following subsections.

\begin{remark}
This result has been extended to any dimension $d\geq 1$ in \cite{CJLV} and it has been proved that
such solutions are equivalent to gradient flow solutions obtained in \cite{Carrillo}.
\end{remark}

\subsection{One-sided Lipschitz estimate}

\begin{lemma}\label{achapoOSL}
Let $\rho(t)\in \calM_b(\RR)$ be nonnegative for all $t\geq 0$.
Then under assumptions \eqref{hyp1} the function
$(t,x)\mapsto \widehat{a}_\rho(t,x)$ defined in \eqref{achapo}
satisfies the one-sided Lipschitz estimate
	$$
\widehat{a}_\rho(t,x)-\widehat{a}_\rho(t,y) \leq \lambda (x-y)|\rho|(\RR),
\quad\mbox{ for all }\ x>y,\ t\geq 0
	$$
\end{lemma}
\begin{proof}
Using assumption \eqref{hyp1}, $x\mapsto W'(x)-\lambda x$ is a nonincreasing function on 
$\RR\setminus\{0\}$. Therefore $\lim_{x\to 0^{\pm}} W'(x) = W'(0^{\pm})$ exists and from 
the oddness of $W'$, we deduce that $W'(0^-)=-W'(0^+)$.
Moreover, for all $x>y$ in $\RR\setminus\{0\}$ we have 
$W'(x)-\lambda x\leq W'(y)-\lambda y$.
Thus we have the one-sided Lipschitz estimate (OSL) for $W'$
\begin{equation}\label{WOSL}
\forall\, x>y\in \RR\setminus\{0\}, \qquad  W'(x)-W'(y)\leq \lambda (x-y).
\end{equation}
Letting $y\to 0^{\pm}$ we deduce that for all $x>0$, 
$W'(x)-\lambda x \leq W'(0^+)$ and $W'(x)-\lambda x \leq W'(0^-)$. 
Thus we also have the one-sided estimate
	\begin{equation}  \label{WOSL1}
W'(x) \leq \lambda x, \quad \mbox{ for all } x>0.  
	\end{equation}

By definition of $\widehat{a}_\rho$ \eqref{achapo}, we have
	$$
\widehat{a}_\rho(x)-\widehat{a}_\rho(y) = \int_{z\neq x, z\neq y} 
(W'(x-z)-W'(y-z))\rho(dz) + W'(x-y) \int_{z\in\{x\}\cup\{y\}} \rho(dz),
	$$
where we use the oddness of $W'$ \eqref{hyp1} in the last term.
Let us assume that $x>y$, from \eqref{WOSL}, we deduce that
$W'(x-z)-W'(y-z)\leq \lambda (x-y)$ and with \eqref{WOSL1}, we deduce
$W'(x-y)\leq \lambda (x-y)$. Thus, using the nonnegativity of $\rho$, 
we deduce the one-sided Lipschitz (OSL) estimate for $\widehat{a}_\rho$.
\end{proof}

\subsection{Dynamics of aggregates}\label{dyna}

We first assume that the initial density is given by a finite sum of Dirac deltas:
$\rho^{ini}_n=\sum_{i=1}^n m_i \delta_{x_i^0}$ 
where $x_1^0<x_2^0<\dots<x_n^0$ and the $m_i$-s are nonnegative.
Moreover, we assume that $\sum_{i=1}^n m_i=1$ and that
the first moment $\sum_{i=1}^n m_i |x_i^0|$ is uniformly bounded with respect to $n$,
so that $\rho^{ini}_n\in \calP_1(\RR)$.
We look for a solution in the form $\rho_n(t,x)=\sum_{i=1}^n m_i \delta_{x_i(t)}$.
Injecting this expression into the definition of the macroscopic velocity in \eqref{achapo}, we get
\begin{equation}\label{achaporhon}
\widehat{a}_{\rho_n}(x) = \left\{\begin{array}{ll}
\ds \sum_{j\neq i} m_jW'(x_i-x_j) & \ds \qquad \mbox{ if } x=x_i,\ i=1,\ldots,n  \\[2mm]
\ds \sum_{j=1}^n m_j W'(x-x_j) & \ds \qquad \mbox{ otherwise}.
\end{array}\right.
\end{equation}
We emphasize that this macroscopic velocity is defined everywhere, which allows 
to give a sense to its product with the measure $\rho_n$.
Then, $\rho_n$ is a solution in the sense of distributions of \eqref{rhodis}
provided the sequence $(x_i)_{i=1,\ldots,n}$ satisfies the ODE system
\begin{equation}\label{EDOxi}
x'_i(t) = \sum_{j\neq i} m_j W'(x_i-x_j),\qquad x_i(0)=x_i^0,\qquad i=1,\ldots,n_\ell,
\end{equation}
where $n_\ell\leq n$ is the number of distinct particles, i.e.
$n_\ell = \#\{i\in \{1,\ldots,N\},  x_i\neq x_j, \forall\,j \}$.
Then we define the dynamics of aggregates by:
\begin{itemize}
\item When the $x_i$ are all distinct, they are solutions of system 
\eqref{EDOxi} (with zero right hand side if $n_\ell =1$).
\item When two particles collide, they stick to form a bigger particle whose mass is the sum of both particles and the dynamics continues with one particle less.
\end{itemize}
Clearly this choice of the dynamics implies mass conservation. It also preserves the one-sided Lipschitz estimate for the velocity.
Finally, setting $\rho_n(t,x) = \sum_{i=1}^{n_\ell} m_i\delta_{x_i(t)}(x)$, the sticky particle dynamics defines a distributional solution to \eqref{rhodis}. 
Hence, we are in position to apply
 Theorem \ref{dual2distrib}, and deduce that $\rho_n(t,x)$ is a duality
solution for given initial data $\rho_n^{ini}$.

For a general initial datum $\rho^{ini}$ in $\calP_1(\RR)$, we approximate it 
by a sequence of measures $\rho^{ini}_n$, for which we can construct a duality solution as above.
Then we use the stability of duality solutions (see Theorem \ref{ExistDuality})
to pass to the limit in the approximation.
This allows to prove the existence result in Theorem \ref{th:duality}.

\subsection{Contraction property}\label{uniq}

Uniqueness in Theorem \ref{th:duality} is obtained thanks to a contraction 
argument in the Wasserstein distance. In the present one dimensional framework, the definition of the
Wasserstein distance can be simplified using the generalized inverse.
More precisely, let $\rho$ be a nonnegative measure, we denote by $F$ its cumulative distribution 
function. Then we can define the generalized inverse of $F$ (or monotone rearrangement of $\rho$)
by $F^{-1}(z):=\inf\{x\in \RR / F(x)>z\}$, it is a right-continuous
and nondecreasing function as well, defined on $[0,1]$.
We have for every nonnegative Borel map $\xi$,
	$$
\int_\RR \xi(x) \rho(dx) = \int_0^1 \xi(F^{-1}(z))\,dz.
	$$
In particular, $\rho\in \calP_1(\RR)$ if and only if $F^{-1}\in L^1(0,1)$.
Then, if $\rho_1$ and $\rho_2$ belong to $\calP_1(\RR)$, with monotone rearrangement
$F_1$ and $F_2$, respectively, we have the explicit expression of the Wasserstein distance (see \cite{Villani2})
\begin{equation}\label{dWF-1}
d_{W1}(\rho_1,\rho_2) = \int_0^1 |F_1(z)-F_2(z)|\,dz.
\end{equation}

Let $\rho$ be a duality solution that satisfies \eqref{rhodis} in the distributional sense.
Denoting $F$ its cumulative distribution function and $F^{-1}$ its generalized inverse,
we have by integration of \eqref{rhodis}
	$$
\pa_t F + \widehat{a}_\rho \pa_x F = 0,
	$$
so that the generalized inverse is a solution to
\begin{equation}
  \label{eq:F-1}
\pa_t F^{-1}(t,z) = \widehat{a}_\rho(t,F^{-1}(z)).  
\end{equation}
Moreover thanks to a change of variables in \eqref{achapo}, 
	$$
\widehat{a}_\rho(t,F^{-1}(z))=\int_{y\neq z} W'\big(F^{-1}(z)-F^{-1}(y)\big)\,dy.
	$$
Now using \eqref{eq:F-1} and since $W'$ is one-sided Lipschitz continuous \eqref{WOSL}, 
we obtain the following contraction property, which implies uniqueness.
\begin{proposition}\label{propuniq}
Assume $\rho_1(t,\cdot),\rho_2(t,\cdot)\in\calP_1(\RR)$ satisfy \eqref{rhodis} in the 
sense of distributions, with $\widehat{a}_{\rho_i}$ given by \eqref{achapo}, and initial data
$\rho_1^{ini}$ and $\rho_2^{ini}$. Then we have, for all $t>0$
	$$
d_{W1}(\rho_1(t,\cdot),\rho_2(t,\cdot)) \leq e^{2\lambda t} d_{W1}(\rho_1^{ini},\rho_2^{ini}).
	$$
\end{proposition}

\section{Existence and uniqueness in the nonlinear case}\label{sec:nonlin}

In this Section, we focus on the nonlinear case characterized  by assumptions \eqref{hyp_a} and \eqref{hyp2}.
The main result is the following theorem, which includes the existence result of \cite{NoDEA}, where 
the particular case $W(x)=\frac 12 e^{-|x|}-\frac 12$ arising in chemotaxis has been considered, and
the existence result presented in \cite{CRAS}, when $W(x)=-|x|/2$, which 
appears in many applications in physics or biology.

	\begin{theorem}\cite[Theorem 3.10]{GF_dual}\label{dual_anonid}
Let be given $\rho^{ini}\in\calP_1(\RR)$.
Let us assume that $a$ and $W$ are such as in \eqref{hyp2}.
For all $T> 0$ there exists
a unique duality solution $\rho$ of \eqref{EqInter} in the sense 
of Definition \ref{defexist} that satisfies in the distributional sense
	\begin{equation}\label{eqrhodis}
\pa_t \rho + \pa_x J = 0,
	\end{equation}
where $J$ is defined by \eqref{DefFluxJ}. Moreover, $\rho(t)\in \calP_1(\RR)$
for $t\in (0,T)$.
	\end{theorem}

The proof follows the same steps as the preceding one, except that uniqueness here relies on an entropy estimate. In this respect,
we  emphasize now some links with entropy solutions of scalar conservation laws.
\begin{remark}
If $w=0$ in \eqref{hyp2}, then $\rho$ is the unique duality solution of Theorem \ref{dual_anonid} if and only if
$u=W'*\rho$ is an entropy solution to $\pa_t u + \pa_x A(u) = 0$. (see \cite[Theorem 5.7]{GF_dual})
\end{remark}

\subsection{One sided-Lipschitz estimate}

\begin{lemma}\label{aOSL}
Assume $0\leq\rho\in {\mathcal M}_b(\RR)$ and that \eqref{hyp2} is satisfied.
Then the function $x\mapsto a(W'*\rho)$ satisfies the OSL condition \eqref{OSLC}.
\end{lemma}
\begin{proof}
Using the assumption on $W$ in \eqref{hyp2}, we deduce that
$$
\pa_{xx}W*\rho = -\rho + w*\rho.
$$
Therefore,
$$
\pa_x(a(\pa_xW*\rho)) = a'(\pa_xW*\rho)(-\rho+w*\rho)
\leq a'(\pa_xW*\rho) w*\rho,
$$
where we use the nonnegativity of $\rho$ in the last inequality.
Then from assumption \eqref{hyp_a} on the function $a$ we get
	$$
\pa_x(a(\pa_xW*\rho)) \leq \alpha \|\rho\|_{L^1} \|w\|_{L^\infty}.
	$$
\end{proof}

\subsection{Dynamics of aggregates}

Following the idea in subsection \ref{dyna}, we first approximate 
the initial data $\rho^{ini}$ by a finite sum of Dirac deltas:
$\rho^{ini}_n=\sum_{i=1}^n m_i \delta_{x_i^0}$ where $x_1^0<x_2^0<\dots<x_n^0$
and the $m_i$ are nonnegative.
We assume that $\sum_{i=1}^n m_i=1$ and 
$\sum_{i=1}^n m_i |x_i^0|$ is uniformly bounded with respect to $n$,
i.e. $\rho^{ini}_n\in \calP_1(\RR)$.
We look for a sequence $(\rho_n)_n$ solving in the distributional sense
$\pa_t\rho_n +\pa_x J_n=0$ where the flux $J_n$ is given by \eqref{DefFluxJ}.
Let $\rho_n(t,x)=\sum_{i=1}^n m_i \delta_{x_i(t)}$.
From assumption \eqref{hyp2} on $W$, we deduce that
$$
W'(x) = - H(x) + \widetilde{w}(x), \quad \mbox{ where }
\widetilde{w}(x) = \int^x_0 w(y)\,dy +\frac 12.
$$
Then, we have
$$
W'*\rho_n(x_i^+) = -\sum_{j=1}^i m_j + 
\sum_{j=1}^n m_j \widetilde{w}(x_i-x_j);
\quad W'*\rho_n(x_i^-) = m_i + W'*\rho_n(x_i^+),
$$
where we use the standard notation $f(x_i^+)=\lim_{x\underset{>}{\to} x_i} f(x)$
and $f(x_i^-)=\lim_{x\underset{<}{\to} x_i} f(x)$.
From these identities, we deduce that in the distributional sense
$$
\pa_x \big(A(W'*\rho_n)\big) = a(W'*\rho_n)w*\rho_n + 
\sum_{i=1}^n [A(W'*\rho_n)]_{x_i} \delta_{x_i},
$$
where $[f]_{x_i}=f(x_i^+)-f(x_i^-)$ is the jump of the function $f$ at $x_i$.
Then we find that $\rho_n$ satisfies \eqref{eqrhodis} in the sense of distributions
if we have
\begin{equation}\label{dynagg}
m_i x'_i(t) = -[A(W'*\rho_n)]_{x_i(t)}, \quad \mbox{ for } i=1,\dots, n_\ell.
\end{equation}
This system of ODEs is complemented by the initial data $x_i(0)=x_i^0$.

Then the dynamics of aggregates is given as in subsection \ref{dyna} by 
\eqref{dynagg} as long as particles are all distinct, and by a sticky dynamics
at collisions.
By construction, $\rho_n(t,x) = \sum_{i=1}^{n_\ell} m_i\delta_{x_i(t)}(x)$ is a duality
solution as in Theorem \ref{dual_anonid} for given initial data $\rho_n^{ini}$.

For a general initial datum $\rho^{ini}$ in $\calP_1(\RR)$, we approximate it 
by a sequence of measures $\rho^{ini}_n$, for which we can construct a duality solution as above.
Then we use the stability of duality solutions (see Theorem \ref{ExistDuality})
to pass to the limit in the approximation.

\subsection{Uniqueness}

The strategy used in subsection \ref{uniq} cannot 
be used here, since it strongly relies on the linearity of $a$.
Then we use an analogy with entropy solutions for scalar conservation laws.
Indeed, the quantity $W'*\rho$ solves a scalar conservation law with source term,
for which we can prove the following entropy estimate:
	\begin{proposition}\cite[Lemma 4.5]{NoDEA}\label{entropy}
Let us assume that assumptions \eqref{hyp2} hold.
For $T>0$, let $\rho\in C([0,T],\calP_1(\RR))$ satisfy \eqref{DefFluxJ}-\eqref{eqrhodis} in the sense of distributions.
Then $u:=W'*\rho$ is a weak solution of 
\begin{equation}\label{eq:U}
\pa_t u+\pa_x A(u) = a(u)w*\rho+\pa_x(w*A(u))-w*(a(u)w*\rho).
\end{equation}
Moreover, if we assume that the entropy condition
\begin{equation}\label{entropycond}
\pa_x u \leq w*\rho
\end{equation}
holds, then for any twice continuously differentiable convex function $\eta$,
we have
\begin{equation}\label{eq:etaU}
\pa_t \eta(u)+\pa_x q(u) -\eta'(u) a(u)w*\rho+\eta'(u)\big(\pa_x(w*A(u))-w*(a(u)w*\rho))
\leq 0,
\end{equation}
where $q$, the entropy flux, is given by $\ds q(x)=\int_0^x \eta'(y)a(y)\,dy$.
	\end{proposition}

This entropy estimate allows to deal with uniqueness.
Consider two solutions $\rho_1$ and $\rho_2$ with initial data
$\rho_1^{ini}$ and $\rho_2^{ini}$, as in Theorem \ref{dual_anonid}.
Since $\rho_1$ and $\rho_2$ are nonnegative, we deduce from \eqref{hyp2}
that both $u_1:=W'*\rho_1$ and $u_2=W'*\rho_2$ satisfy \eqref{entropycond}.
Starting from the entropy inequality \eqref{eq:etaU} with the family 
of Kru\v{z}kov entropies $\eta_\kappa(u) = |u-\kappa|$ and using the
doubling of variable technique developed by Kru\v{z}kov, 
we show
$$
\begin{array}{ll}
\ds \frac{d}{dt} \int_\RR \big|u_1-u_2\big| \leq & 
\|w\|_{Lip} \int_\RR \big|A(u_1)-A(u_2)\big|\,dx   \\
& \ds + \big(1+\|w\|_\infty\big) \int_\RR \big|a(u_1)w*\rho_1-a(u_2)w*\rho_2\big|\,dx.
\end{array}
$$
From \eqref{hyp_a} and the bound on $\rho(t)$ in $\calP_1(\RR)$ for all $t$, 
we deduce that $u_i$, $i=1,2$ are bounded in $L^\infty_{t,x}$.
Then we get
\begin{equation}\label{uniqU}
\frac{d}{dt} \int_\RR \big|u_1-u_2\big| \leq C\Big(
\int_\RR \big|u_1-u_2\big|\,dx + \int_\RR \big|w*\rho_1-w*\rho_2\big|\,dx\Big),
\end{equation}
where we use once again assumption on $a$ in \eqref{hyp_a}.
Taking the convolution with $w$ of equation \eqref{eqrhodis} we deduce
$$
\pa_t w*\rho_i - \pa_x\big(w*A(u_i)\big) + w*\big(a(u_i)w*\rho_i\big)=0, \qquad i=1,2.
$$
We deduce from \eqref{hyp_a} and the Lipschitz bound on $w$ that
\begin{equation}\label{uniqw}
\frac{d}{dt} \int_\RR \big|w*\rho_1-w*\rho_2\big| \leq C\Big(
\int_\RR \big|u_1-u_2\big|\,dx + \int_\RR \big|w*\rho_1-w*\rho_2\big|\,dx\Big).
\end{equation}
Summing \eqref{uniqU} and \eqref{uniqw}, we deduce applying a Gronwall lemma that
for all $T>0$ there exists a nonnegative constant $C_T$ such that for all $t\in [0,T]$,
$$
\begin{array}{l}
\ds \int_{\RR} \Big(\big|W'*\rho_1-W'*\rho_2\big|(t)+\big|w*\rho_1-w*\rho_2\big|(t)\Big)\,dx   \\
\ds \leq C_T
\int_{\RR} \Big(\big|W'*\rho_1^{ini}-W'*\rho_2^{ini}\big|+\big|w*\rho_1^{ini}-w*\rho_2^{ini}\big|\Big)\,dx.
\end{array}
$$
The uniqueness follows easily.

\section{Numerical simulations}\label{num}

An important advantage of the approach presented in this paper, is that it allows
to prove convergence of well designed finite volume schemes. 
Numerical simulations of duality solutions for linear scalar conservation laws with
discontinuous coefficients have been proposed and analyzed in \cite{GJ}.
In the present nonlinear context, a careful discretization
of the velocity has to be implemented in order to recover the dynamics of
aggregates after blow up time.
In \cite{sinum} a finite volume scheme of Lax-Friedrichs type is proposed and analyzed.
Up to our knowledge, this is the only example of numerical scheme allowing to recover
the dynamics of measure solutions after blow up.
In this paper, we perform the same analysis on an upwind-type scheme, which is less diffusive than the Lax-Friedrichs
scheme of \cite{sinum}. Numerical simulations are also proposed.

\subsection{Upwind finite volume scheme}

Let us consider a mesh with constant space step $\Delta x$, and denote $x_i=i \Delta x$ for $i\in \ZZ$.
We fix a constant  time step $\Delta t$, and set $t_n=n \Delta t$ for $n\in \NN$.
For a given nonnegative measure $\rho^{ini}\in \calP_1(\RR)$, we define for $i\in \ZZ$,
	\begin{equation}\label{disrho0}
\rho_{i}^0= \frac{1}{\Delta x}\int_{C_{i}} \rho^{ini}(dx)\geq 0.
	\end{equation} 
Since $\rho^{ini}$ is a probability measure, the total mass of the system is 
$\sum_{i} \Delta x\rho_{i}^0 = 1$.
Assuming that an approximating sequence $(\rho_{i}^n)_{i\in \ZZ}$ is known at time $t_n$, we
compute the approximation at time $t_{n+1}$ by:
\begin{equation}\label{dis_num}
\rho_{i}^{n+1} = \rho_{i}^n - \frac{\Delta t}{\Delta x}
\big(({a}^n_{i})_+^{} \rho_{i}^n + ({a}^n_{i+1})_-^{} \rho_{i+1}^n 
-({a}^n_{i-1})_+^{} \rho_{i-1}^n - ({a}^n_{i})_-^{} \rho_{i}^n \big) 
\end{equation}
The notation $(a)_+ = \max\{0,a\}$ stands for the positive part of the real $a$
and respectively $(a)_- = \min\{0,a\}$ for the negative part.
Then we define the flux by
\begin{equation}\label{eq:Jdis}
J_{i+1/2}^n = ({a}^n_{i})_+^{} \rho_{i}^n + ({a}^n_{i+1})_-^{} \rho_{i+1}^n
\end{equation}

A key point is the definition of the discrete velocity which
should be done in accordance with \eqref{achapo} in the linear case
and with \eqref{DefFluxJ} in the nonlinear case.

{\bf In the linear case}, the discretization of the velocity is given by 
\begin{equation}\label{def:ai}
a_i^n = \achapo(t_n,x_i) = \sum_{j\neq i} W'(x_i-x_j) \rho_j^n \Delta x.
\end{equation}

{\bf In the nonlinear case}, we define
\begin{equation}\label{eq:ai}
 a_{i}^n = \left\{\begin{array}{ll}
 a(\pa_xS_{i+1/2}), &  \mbox{ if }\ \pa_xS_{i+1/2}^n= \pa_xS_{i-1/2}^n,  \\[2mm]
\displaystyle \frac{A(\pa_xS_{i+1/2}^n)-A(\pa_xS_{i-1/2}^n)}{\pa_xS_{i+1/2}^n-\pa_xS_{i-1/2}^n}, \quad & \mbox{ otherwise.}
\end{array}\right.
\end{equation}
In this definition we have set $S_i^n$ an approximation of $W*\rho(t_n,x_i)$. 
Using \eqref{hyp2}, $S_i^n$ is a solution to
\begin{equation}\label{eq:Vwdis}
-\frac{\pa_xS_{i+1/2}^n-\pa_xS_{i-1/2}^n}{\Delta x}+\nu_i^n = \rho_i^n, \qquad
\pa_x S_{i+1/2}^n = \frac{S_{i+1}^n-S_i^n}{\Delta x}.
\end{equation}
The quantity $(\nu_i^n)_i$ corresponds to an approximation of $(w*\rho(t_n,x_i))_i$.
We will use the following discretization
\begin{equation}\label{eq:nu}
\nu_i^n = \sum_{k\in \ZZ} \rho_k^n w_{ki} \Delta x, \quad \mbox{ with } \quad 
\frac 12 (w_{ki}+w_{ki+1})=\frac{1}{\Delta x}\int_{x_i}^{x_{i+1}} w(x-x_k)\,dx.
\end{equation}
Then using \eqref{eq:ai} and \eqref{eq:Vwdis} we recover the discretization of the product
\begin{equation}\label{proddis}
a_i^n \rho_i^n = a_{i}^n \Big(-\frac{\pa_xS_{i+1}^n-\pa_xS_{i}^n}{\Delta x} + \nu_i^n\Big)
 = -\frac{A(\pa_xS_{i+1/2}^n)-A(\pa_xS_{i-1/2}^n)}{\Delta x} + a_i^n \nu_i^n.
\end{equation}

\begin{remark}
If we choose for the nonlinear function $a$ the identity function $a(x)=x$, then we can verify 
that \eqref{eq:ai} reduces to \eqref{def:ai}. Indeed, in this case, we have $A(x)=x^2/2$, such that
\eqref{eq:ai} reduces to $a_i^n = \frac 12 (\pa_xS_{i+1/2}^n+\pa_xS_{i-1/2}^n)$.
Moreover, from \eqref{eq:Vwdis}, we deduce that 
\begin{equation}\label{eqrem1}
\pa_xS_{i+1/2}^n+\pa_xS_{i-1/2}^n = -\sum_{j\neq i} \sgn(x_i-x_j)(\rho_j^n-\nu_j^n)\Delta x.
\end{equation}
From our choice of discretization in \eqref{eq:nu}, we have
	$$
\sum_{j\neq i} \sgn(x_i-x_j)\nu_j^n = \sum_{k\in \ZZ} \rho_k^n \left(\sum_{j<i} w_{kj} -\sum_{j>i} w_{kj}\right)\Delta x.
	$$
And by definition of $w_{kj}$ in \eqref{eq:nu}, we obtain
	$$
\sum_{j<i} w_{kj} -\sum_{j>i} w_{kj} = \frac{1}{\Delta x} \left(\int_{-\infty}^{x_i} w(x-x_k) dx
- \int_{x_i}^{+\infty} w(x-x_k)dx\right).
	$$
Denoting by $\widetilde{w}(x)=\frac{1}{2}(\int_{-\infty}^{x} w(y) dy - \int_{x}^{+\infty} w(y)dy)$,
which is an antiderivative of $w$, we deduce from \eqref{eqrem1} that
	$$
\frac 12 (\pa_xS_{i+1/2}^n+\pa_xS_{i-1/2}^n) = \sum_{j\neq i} \left(\frac 12 \sgn(x_i-x_j)
+\widetilde{w}(x_i-x_j)\right) \rho_j^n \Delta x.
	$$
From \eqref{hyp2}, we deduce that
$W'(x)=-\frac 12 \sgn(x)+\widetilde{w}(x)$ so that
	$$
a_i^n = \frac 12 (\pa_xS_{i+1/2}^n+\pa_xS_{i-1/2}^n) = \sum_{j\neq i} W'(x_i-x_j)\rho_j^n \Delta x,
	$$
which is equation \eqref{def:ai}.
\end{remark}

\subsection{Properties of the scheme}

The following Lemma states a Courant-Friedrichs-Lewy (CFL)-like condition for the scheme.
\begin{lemma}\label{lem:CFL}
Let $\rho^{ini}\in \calP_1(\RR)$.
We define $\rho_{i}^0$ by \eqref{disrho0}.
Let us assume that the following condition is satisfied:
\begin{equation}\label{CFL}
a_\infty \frac{\Delta t}{\Delta x} \leq 1,
\end{equation}
where 
	$$
a_\infty := \left\{\begin{array}{ll}
\ds \hfill\|W\|_{Lip},& \mbox{ in the linear case,}  \\\\
\ds \max_{x\in [-(1+w_0),(1+w_0)]}|a(x)|, & \mbox{ in the nonlinear case.}
\end{array}\right.
	$$
Then the sequence computed thanks to the scheme defined in \eqref{dis_num}
satisfies $\rho_{i}^n \geq 0$ and $|a_i^n|\leq a_\infty$, for all $i\in\ZZ$ and $n\in\NN$, .
\end{lemma}
\begin{proof}
The total initial mass of the system is $\sum_{i} \rho_{i}^0 \Delta x=1$. 
Since the scheme \eqref{dis_num} is conservative, we have for all $n\in \NN$, 
$\sum_{i} \rho_{i}^n\Delta x=1$.

We can rewrite equation \eqref{dis_num} as
\begin{equation}
\rho_{i}^{n+1} = \rho_{i}^n \left( 1 - \frac{\Delta
t}{\Delta x} |{a}^n_{i}|\right) - \rho_{i+1}^n \frac{\Delta t}{\Delta x}({a}^n_{i+1})_-^{}
+ \rho_{i-1}^n \frac{\Delta t}{\Delta x}({a}^n_{i-1})_+^{}.
\label{schemarho}
\end{equation}

Then assuming that condition \eqref{CFL} holds, we prove the Lemma by induction on $n$. 
For $n=0$, we have $\rho_i^0\geq 0$ from \eqref{disrho0} and from \eqref{def:ai} 
in the linear case, we deduce that
$|a_i^0|\leq a_\infty$. In the nonlinear case, we have from \eqref{eq:Vwdis}
$$
|\pa_xS^0_{i+1/2}| = \Delta x |\sum_{k\leq i} (\nu_k^0-\rho_k^0)| \leq 1+w_0,
$$
Indeed from \eqref{eq:nu}, $\sum_i |\nu_i^0| \leq w_0$,
where $w_0$ is defined in \eqref{hyp2}.
Then using definition \eqref{eq:ai}, we obtain by the mean value theorem that 
$|a_i^0|\leq a_\infty$. Thus the result is satisfied for $n=0$.

By induction, we assume that the estimates hold for some $n\in \NN$.
Then, in the scheme \eqref{schemarho}, all the coefficients in front of $\rho_{i}^n$, 
$\rho_{i-1}^n$ and $\rho_{i+1}^n$ are nonnegative.
We deduce that the scheme is nonnegative therefore $\rho_i^{n+1} \geq 0$ for all $i\in \ZZ$ 
and $n\in \NN$.
Next, in the linear case, from \eqref{def:ai} and \eqref{hyp_a} we deduce that $|a_i^{n+1}|\leq a_\infty$;
in the nonlinear case, we have from \eqref{eq:Vwdis} and the mass conservation
$$
|\pa_xS^{n+1}_{i+1/2}| = \Delta x |\sum_{k\leq i} (\nu_k^{n+1}-\rho_k^{n+1})| \leq 1+w_0.
$$
As above, from the definition \eqref{eq:ai} and the mean value theorem, we deduce that
$|a_i^{n+1}|\leq a_\infty$.
\end{proof}

In the following Lemma, we gather some properties of the scheme.

\begin{lemma}\label{bounddismom}
Let $\rho_{i}^0$ defined by \eqref{disrho0} for some $\rho^{ini}\in \calP_1(\RR)$.
Let us assume that \eqref{CFL} is satisfied.
Then the sequence $(\rho_{i}^n)$ constructed thanks to the numerical scheme
\eqref{dis_num} satisfies:

$(i)$ Conservation of the mass:
for all $n\in \NN^*$, we have
\begin{eqnarray*}
&\ds \sum_{i\in \ZZ} \rho_{i}^n \Delta x=
\sum_{i\in \ZZ} \rho_{i}^0 \Delta x = 1\ , \\
\end{eqnarray*}

$(ii)$ Bound on the first moment: there exists a constant $C>0$ such that for all $n\in \NN^*$, we have
\begin{equation}\label{boundmoment1}
M_1^n := \sum_{i\in \ZZ} |x_i| \rho_{i}^n \Delta x \leq M_1^0 + a_\infty t_n.
\end{equation}
where $t_n=n\Delta t$.

$(iii)$ Support: Let us assume that $a_\infty\Delta t = \gamma \Delta x$ for $\gamma \in (0,1)$.
If $\rho^{ini}$ has a finite support then the numerical approximation at finite time
has a finite support too.
More precisely, assuming that $\rho^{ini}$ is compactly supported in $B(0,R)$,
then for any $T>0$, we have $\rho^n_i=0$ for any $n\leq T/\Delta t$ and any $i\in \ZZ$
such that $x_i \notin B(0,R+\frac{a_\infty T}{\gamma})$.
\end{lemma}

\begin{proof}

$(i)$ It is a direct consequence of Lemma \ref{lem:CFL} and the fact that the scheme is conservative.

$(ii)$ For the first moment, we have from \eqref{dis_num} after using a discrete integration by parts:
$$
\begin{array}{ll}
\ds \sum_{i\in \ZZ} |x_i|\rho_{i}^{n+1} = & \ds \sum_{i\in
\ZZ} |x_i|\rho_{i}^n  \\
& \ds - \frac{\Delta t}{\Delta x} \sum_{i\in \ZZ} 
\left(({a}^n_{i})_+ \, \rho_{i}^n \big(|x_i|-|x_{i+1}|\big)
+({a}^n_{i})_- \, \rho_{i}^n \big(|x_{i-1}|-|x_{i}|\big) \right).
\end{array}
$$
From the definition $x_i=i\Delta x$, we deduce
$$
\sum_{i\in \ZZ} |x_i|\rho_{i}^{n+1}\Delta x \leq \sum_{i\in \ZZ}
|x_i|\rho_{i}^n \Delta x + \Delta t \sum_{i\in \ZZ} |a^n_{i}| \, \rho_{i}^n \Delta x.
$$
Since the velocity is bounded by $a_\infty$ from Lemma \ref{lem:CFL}, 
and using the mass conservation we get $M_1^{n+1} \leq M_1^n + a_\infty \Delta t$.
The conclusion follows directly by induction on $n$.

$(iii)$ By definition of the numerical scheme \eqref{dis_num}, we clearly have that 
the support increases of only $1$ point of discretisation at each time step. 
Therefore after $n$ iterations, the support has increased of 
$n\Delta t= n \Delta t a_\infty/\gamma \leq T a_\infty/\gamma$.

\end{proof}

\subsection{Convergence result}

We define the initial reconstruction 
\begin{equation}\label{rhoh} 
\rho_h^n(x) = \sum_{i\in \ZZ} \rho_{i}^n \Delta x \delta_{x_i}.
\end{equation}
Then we construct $\rho_h(t,x) = \sum_{n=0}^{N_T} \rho_h^n(x) {\mathbf 1}_{[t_n,t_{n+1})}$,
where $N_T = T/\Delta t$.
The following result proves the convergence of this approximation towards the unique duality
solution of the aggregation equation.
A similar result for the Lax-Friedrichs scheme has been proved in \cite[Theorem 3.4]{sinum}.

\begin{theorem}\label{th:convmacro}
Let us assume that $\rho^{ini}\in \calP_1(\RR)$ is given,
compactly supported and nonnegative and define $\rho_i^0$ by \eqref{disrho0}.
\begin{itemize}
\item {\it In the linear case}, under assumptions \eqref{hyp1} and $a={\rm id}$, if (\ref{CFL}) is satisfied, 
then the discretization $\rho_h$ converges in $\calM_b([0,T]\times \RR)$ towards the 
unique duality solution $\rho$ of Theorem \ref{th:duality} 
as $\Delta t$ and $\Delta x$ go to $0$.
\item {\it In the nonlinear case}, under assumptions \eqref{hyp_a} and \eqref{hyp2}, if (\ref{CFL}) is satisfied,
then the discretization $\rho_h$ converges in $\calM_b([0,T]\times \RR)$ towards the 
unique duality solution $\rho$ of Theorem \ref{dual_anonid} 
as $\Delta t$ and $\Delta x$ go to $0$.
\end{itemize}
\end{theorem}

\begin{proof}
The proof follows closely the ideas of \cite[Theorem 3.4]{sinum}, which are adapted to
the upwind scheme.
First we define $M_h(t,x)=\int_{-\infty}^x \rho_h(t,dy)$. This is a piecewise constant function:
on $[t_n,t_{n+1})\times[x_i,x_{i+1})$ we have $M_h(t)=M_i^n:= \sum_{k\leq i} \rho_k^n \Delta x$.
After a summation of \eqref{dis_num}, we deduce
	$$
M_i^{n+1} = M_i^n \left(1-\Big((a_i^n)_+^{}-(a_{i+1}^n)_-^{}\Big)\frac{\Delta t}{\Delta x}\right)
+ (a_i^n)_+^{} \frac{\Delta t}{\Delta x} M_{i-1}^n - (a_{i+1}^n)_-^{} \frac{\Delta t}{\Delta x} M_{i+1}^n.
	$$
Introducing the incremental coefficients as in Harten and Le Roux \cite{LeRoux,Harten} 
	$$
b_{i+1/2}^n = -(a_{i+1}^n)_-\frac{\Delta t}{\Delta x}; \qquad 
c_{i-1/2}^n = (a_i^n)_+\frac{\Delta t}{\Delta x},
	$$
we can rewrite the latter equation as
	$$
M_i^{n+1} = M_i^n + c_{i-1/2}^n (M_{i-1}^n-M_i^n) + b_{i+1/2}^n (M_{i+1}^n-M_i^n).
	$$
Provided the CFL condition of Lemma \ref{lem:CFL} is satisfied, we have
$0\leq b_{i+1/2}$, $0\leq c_{i+1/2}$ and $b_{i+1/2}+c_{i+1/2}\leq 1$, so that following \cite{LeRoux,Harten}
 the scheme is TVD provided the CFL condition holds.

It is now standard to prove prove a total variation in time which will imply a $BV([0,T]\times \RR)$ bound. Then we apply the Helly compactness Theorem 
to extract a subsequence of $M_h$ 
converging in $L^1_{loc}([0,T]\times \RR)$ towards some function $\widetilde{M}$.
Next we use a diagonal extraction procedure to extend the local convergence to the whole real line.
We refer the reader for instance to \cite{GR} for more details on these well-known techniques.
We deduce the convergence of $\rho_h$ towards $\rho:=\pa_x\widetilde{M}$ in $\calM_b([0,T]\times \RR)$.

By definition of $\rho_h^n$ in \eqref{rhoh}, we have for any test function $\phi$,
$$
\int_{\RR} \phi(x) \rho_h^{n+1}(dx) = \sum_{i\in\ZZ} \rho_i^{n+1} \phi(x_i) \Delta x.
$$
Then from the definition of the scheme \eqref{dis_num} and using a discrete integration
by parts, we get
$$
\begin{array}{l}
\ds \int_{\RR} \phi(x) \rho_h^{n+1}(dx) = \\
= \ds  \sum_{i\in\ZZ} \rho_i^{n} \Delta x \left[
\phi(x_i) -\frac{\Delta t}{\Delta x}\Big( (a_i^n)_+^{} \big(\phi(x_i)-\phi(x_{i+1})\big)
+(a_i^n)_-^{} \big(\phi(x_{i-1})-\phi(x_{i})\big)\Big)\right].
\end{array}
$$
Using a Taylor expansion, there exist $y_i\in (x_i,x_{i+1})$ and $z_i\in (x_{i-1},x_i)$
such that
$$
\int_{\RR} \phi(x) \rho_h^{n+1}(dx) = \sum_{i\in\ZZ} \rho_i^{n} \Delta x \left[
\phi(x_i) + a_i^n \phi'(x_i) \Delta t + \frac 12 \Delta t \Delta x \big(\phi''(y_i) + \phi''(z_i)\big)\right].
$$
Let us denote $a_h$ an affine interpolation of the sequence $(a_i^n)$ such that $a_h(t_n,x_i)=a_i^n$.
We have
$$
\begin{array}{ll}
\ds \int_{\RR} \phi(t,x) \frac{(\rho_h^{n+1}-\rho_h^n)(t,dx)}{\Delta t} = 
& \ds \int_{\RR} \phi'(t,x) a_h(t,x) \rho_h(t,dx)  \\
& \ds + \frac 12  \sum_{i\in\ZZ} \rho_i^{n} \Delta x^2\big(\phi''(y_i) + \phi''(z_i)\big).
\end{array}
$$
Passing to the limit in the latter identity, using Lemma \ref{lemstab_a} in the linear case
or \eqref{proddis} in the nonlinear case, we deduce that the limit $\rho$ satisfies in the sense 
of distributions equation \eqref{rhodis}. By uniqueness of such a solution, we deduce that 
$\rho$ is the unique duality solution in Theorem \ref{th:duality} in the linear case, or respectively
the unique duality solution in Theorem \ref{dual_anonid}, and the whole sequence converges.
\end{proof}

\subsection{Numerical results}
\label{sec:numsim}

To illustrate the behaviour of solutions, we propose numerical simulations obtained with 
scheme \eqref{dis_num} for three examples of interacting potential.
In these examples we choose the computational domain $[-2.5,2.5]$ discretized with 1000 intervals.
The time step is chosen in order to satisfy the CFL condition \eqref{CFL}.
We consider two initial data. In Figures \ref{fig1}, \ref{fig2} and \ref{fig3} Left, 
the initial data is given by a sum of two bumps:
\begin{equation}\label{init1}
\rho^{ini}(x)=\exp(-10(x-0.7)^2)+\exp(-10(x+0.7)^2).
\end{equation}
In Figures \ref{fig1}, \ref{fig2} and \ref{fig3} Right, 
the initial data is given by a sum of three bumps:
\begin{equation}\label{init2}
\rho^{ini}(x)=\exp(-10(x-1.25)^2)+0.8 \exp(-20x^2)+\exp(-10(x+1)^2).
\end{equation}

{\bf Example 1:}
Figure \ref{fig1} displays the time dynamics of the density $\rho$ if 
we take $W(x)=\frac 12 e^{-|x|}-\frac 12$ and $a(x)=\frac{2}{\pi} \mbox{atan}(50 x)$.
Figure \ref{fig1}-left gives the dynamics for the initial data in \eqref{init1}.
We observe the blow up into two numerical Dirac deltas in a very short time.
Then the two Dirac deltas aggregate into one single Dirac mass which is stationary.
A similar phenomenon is observed on Figure \ref{fig1}-right where the dynamics for the
initial data \eqref{init2} is plotted.

\begin{figure}[ht!]
\begin{center}
  \includegraphics[width=2.47in]{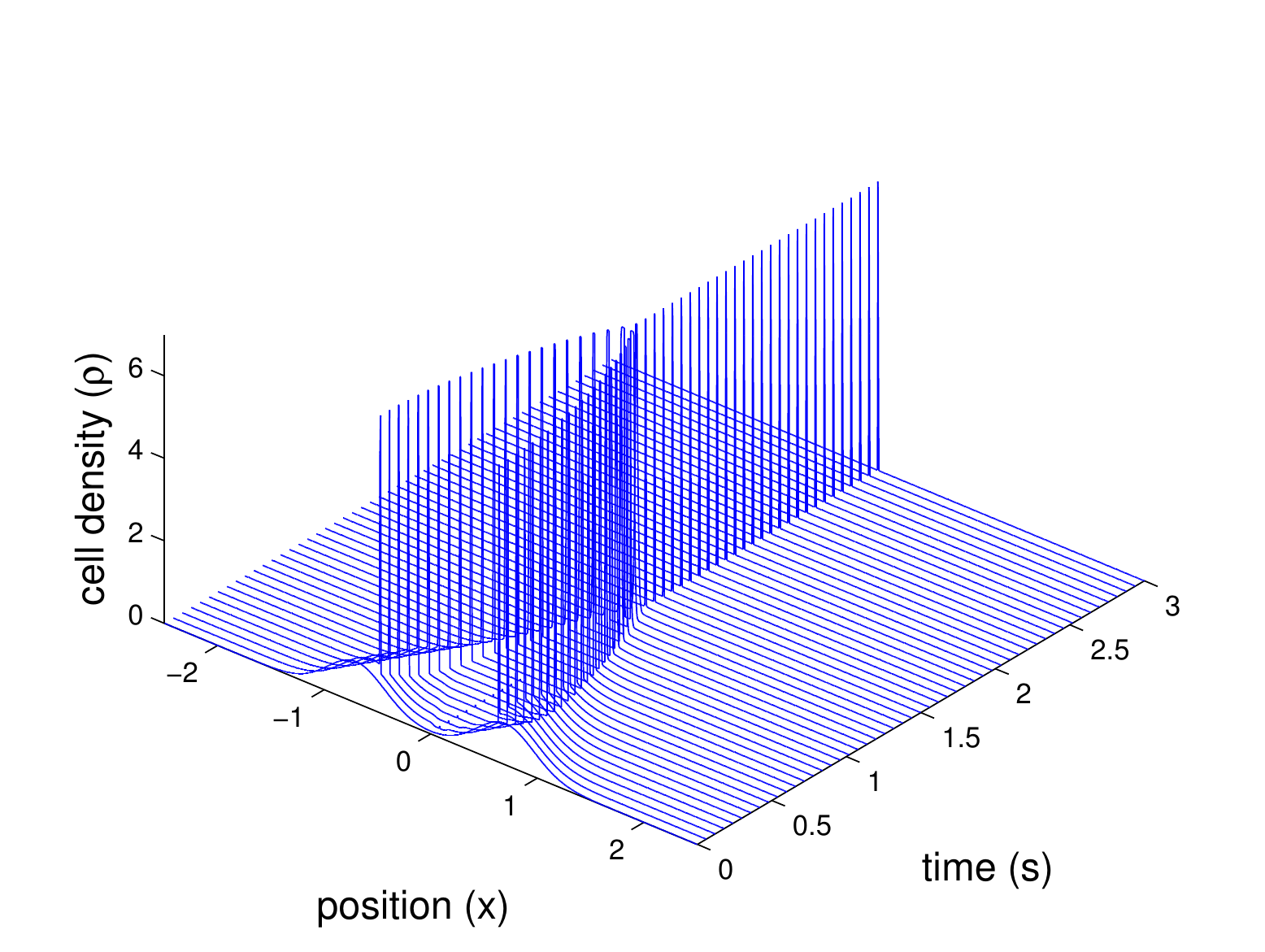}
  \includegraphics[width=2.47in]{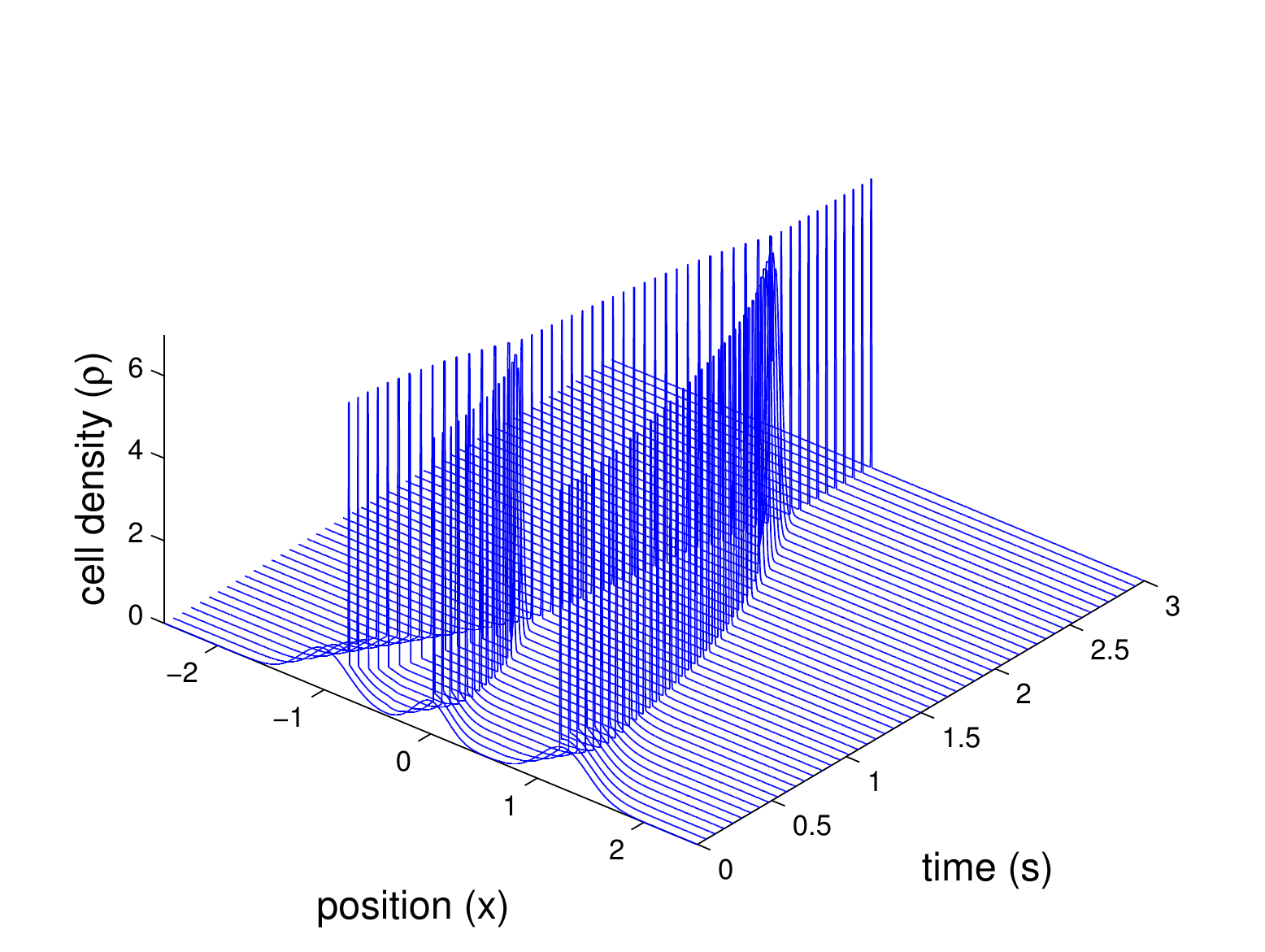}\\
  \caption{Dynamics of the density $\rho$ in the case $W=\frac 12 (e^{-|x|}-1)$
and for a nonlinear function $a(x)=\frac{2}{\pi} \mbox{atan}(50 x)$.}\label{fig1}
\end{center}
\end{figure}

{\bf Example 2:}
In Figure \ref{fig2} we display the time dynamics of the density $\rho$ for
$W(x)=-\frac{|x|}{250}$ and $a(x)=\frac{2}{\pi} \mbox{atan}(50 x)$.
Contrary to the first example, we observe that the blow up occurs in the center and all
bumps concentrate in this point to form a Dirac delta.

\begin{figure}[ht!]
\begin{center}
  \includegraphics[width=2.47in]{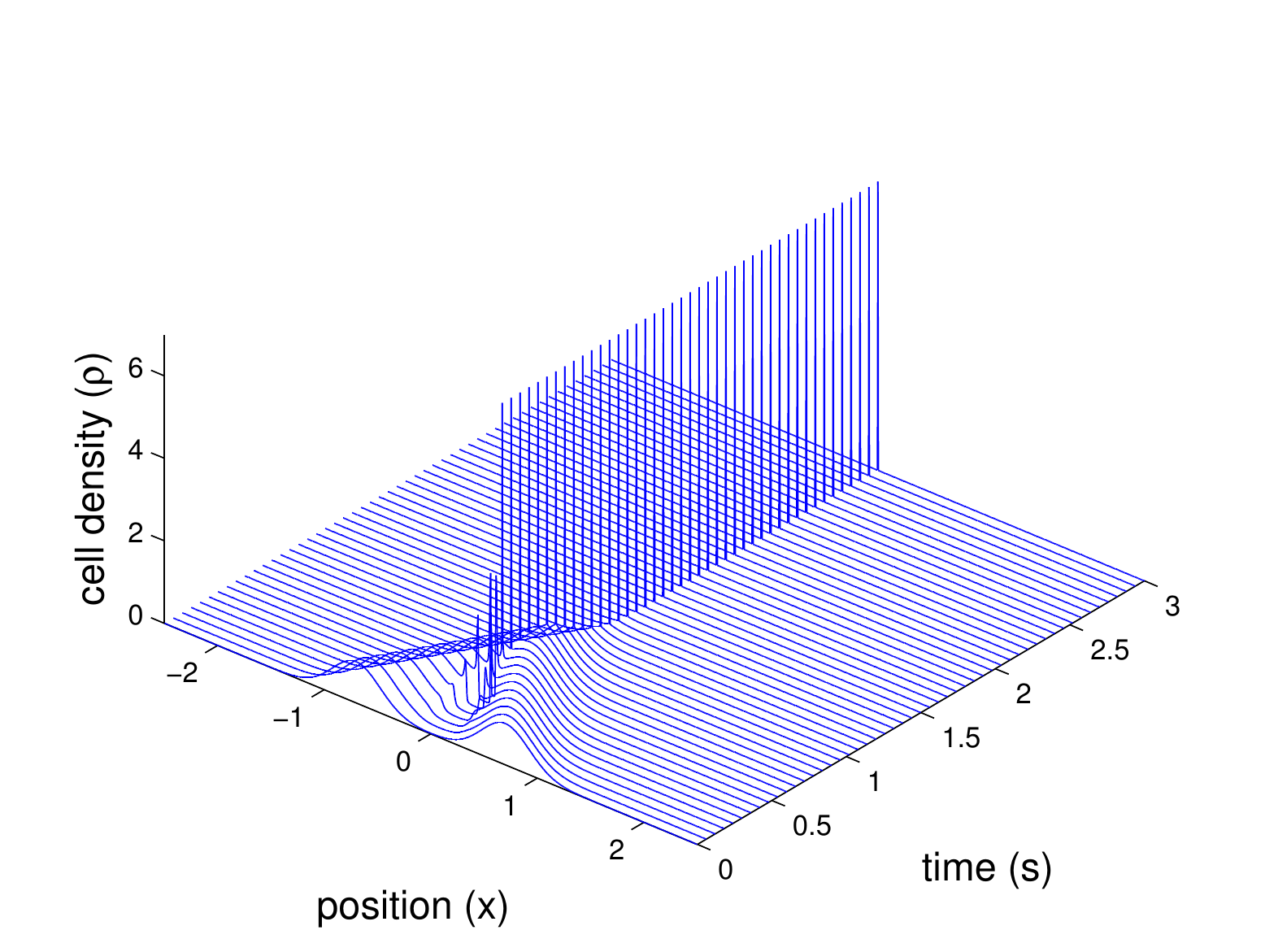}
  \includegraphics[width=2.47in]{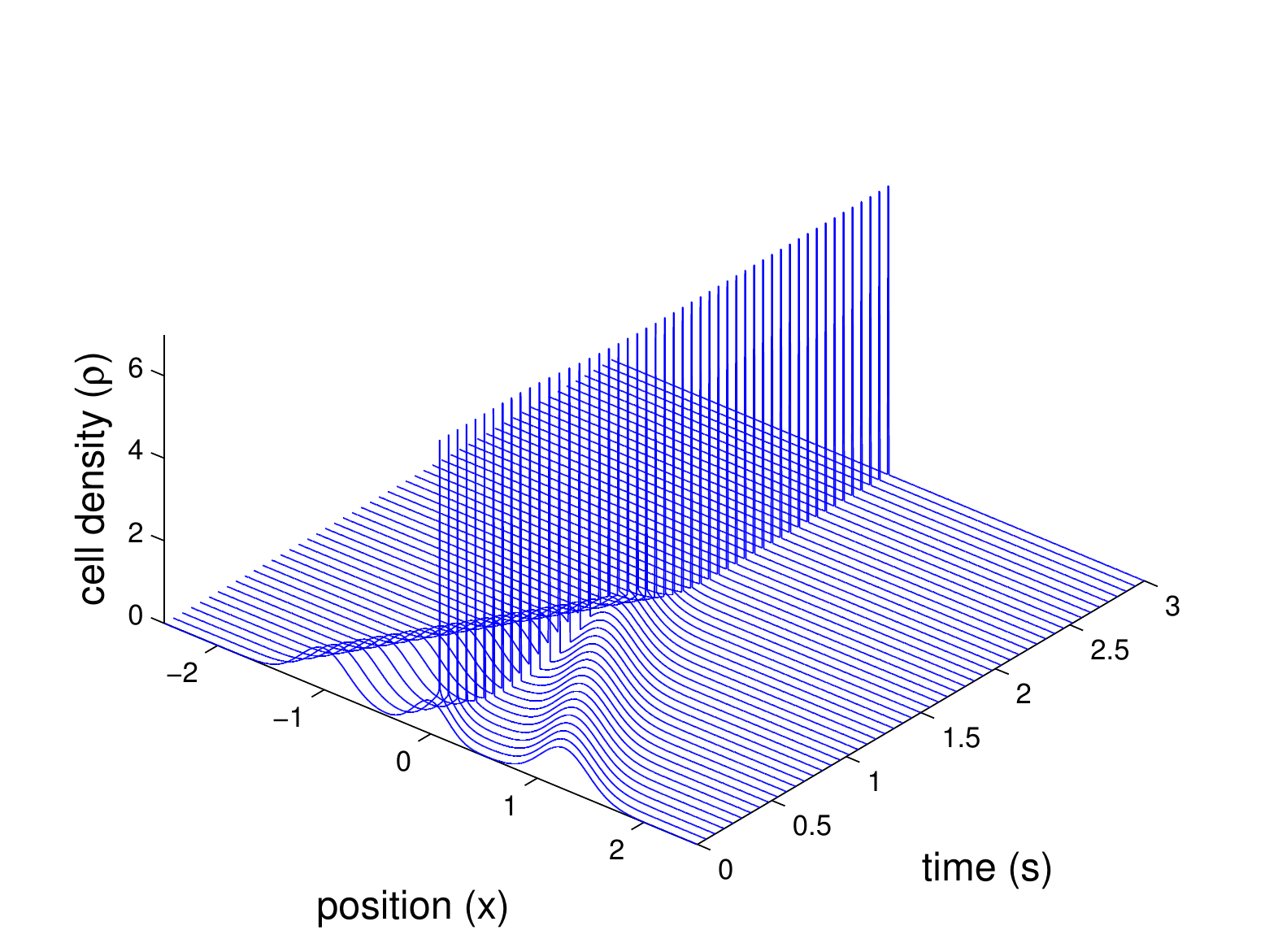}\\
  \caption{Dynamics of the density $\rho$ in the case $W=-\frac{|x|}{250}$
and for a nonlinear function $a(x)=\frac{2}{\pi} \mbox{atan}(50 x)$.}\label{fig2}
\end{center}
\end{figure}

{\bf Example 3:}
Finally, Figure \ref{fig3} displays the time dynamics of the density $\rho$ when
$W(x)=-\frac{|x|}{250}$ and in the linear case $a(x)=x$.
In this last example, the bumps attract themselves and blow up in the same time.
Then in Figure \ref{fig3}-left, with initial data \eqref{init1}, the blow up occurs when the two initial bumps are close to each other.
In Figure \ref{fig3}-right, with initial data \eqref{init2}, the bump in the center blows up before 
the external ones.

\begin{figure}[ht!]
\begin{center}
  \includegraphics[width=2.47in]{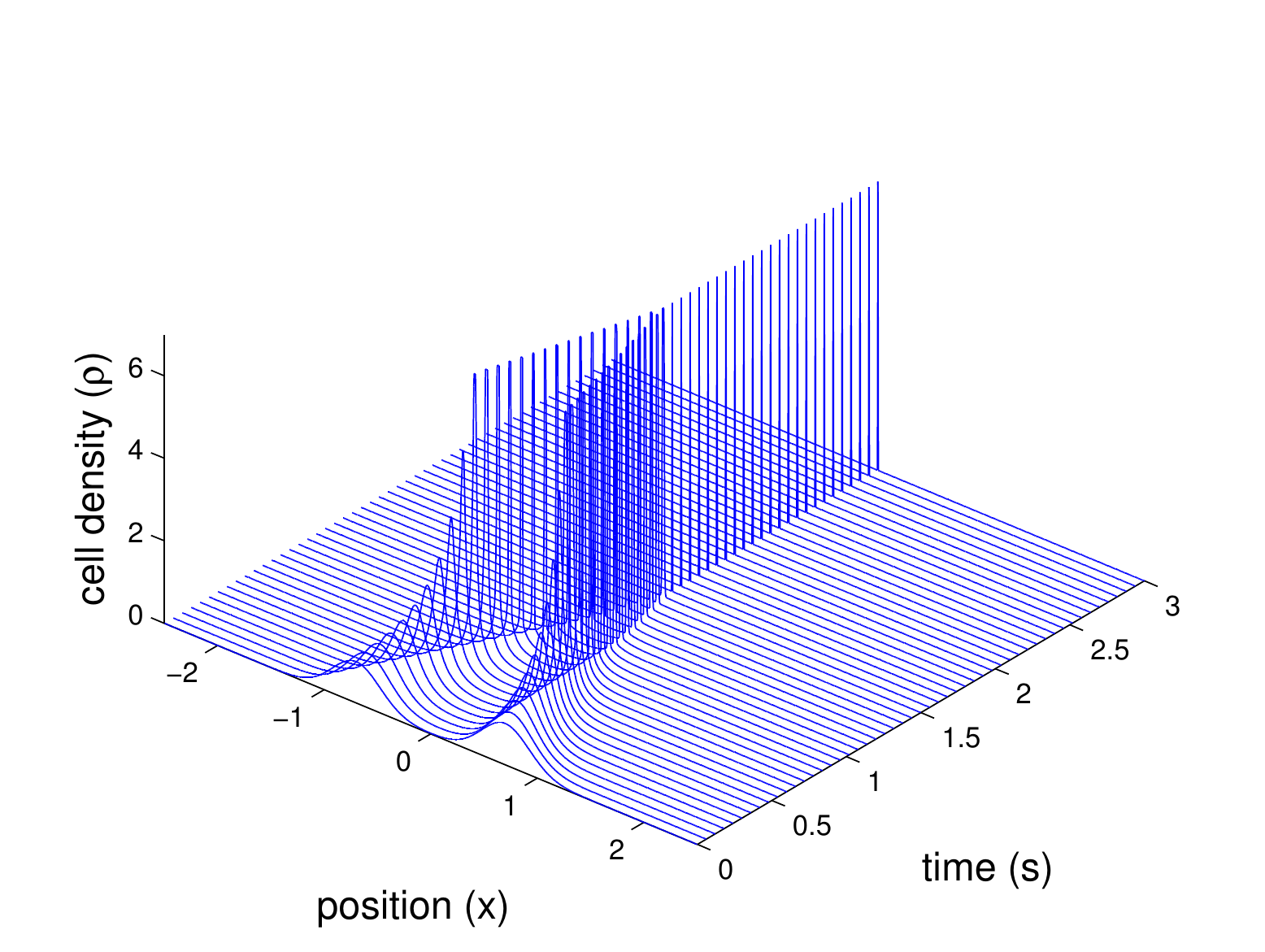}
  \includegraphics[width=2.47in]{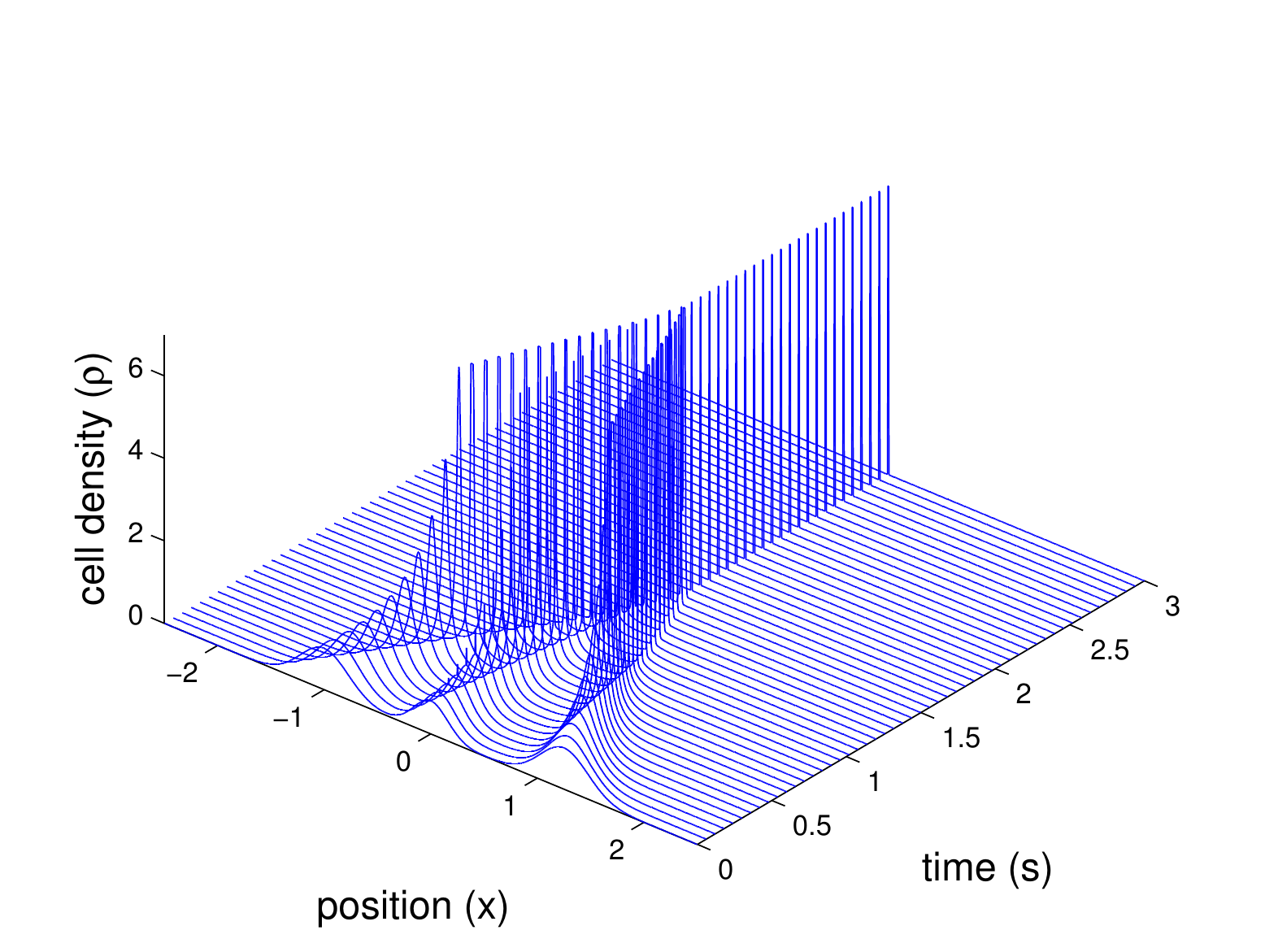}\\
  \caption{Dynamics of the density $\rho$ in the case $W=-\frac{|x|}{250}$
and when $a$ is linear.}\label{fig3}
\end{center}
\end{figure}



\end{document}